\documentclass[oneside,a4paper,11pt]{article}

\usepackage{times}
\fontfamily{ptm}\selectfont

\newcommand{\thetitle}{Distribution of the eigenvalues of a random system of homogeneous polynomials}

\newcommand{\thelanguage}{english}

\title{\thetitle}

\usepackage[utf8]{inputenc}
\usepackage[onehalfspacing]{setspace}
\usepackage[T1]{fontenc}
\usepackage[\thelanguage]{babel}
\usepackage[babel]{csquotes}
\usepackage{tikz}
\usepackage{tikz-3dplot}
\usepackage{mathrsfs}
\usepackage{rotating}
\usepackage{float}
\usepackage{amsmath,amssymb}
\usepackage[amsmath,thmmarks,standard]{ntheorem}
\usepackage[colorlinks=true, citecolor=blue]{hyperref}
\usepackage[\thelanguage]{cleveref}
\usepackage{mathdots}
\usepackage{color,calc,comment}
\usepackage{multirow}
\usepackage{enumitem}
\usepackage{tabularx}
\usepackage{graphicx}
\usepackage[all,2cell]{xy}  
\usepackage{mathtools}
\usepackage{stmaryrd} 
\usepackage{ifpdf}
\usepackage[h]{esvect}
\usepackage{bbm}
\usepackage{marvosym}
\usepackage{blkarray}
\usepackage{ifdraft}
\usepackage{chngcntr}
\usepackage{listings}
\usepackage{fancyhdr}
\usepackage{algorithm,algorithmic}
\usepackage{polynom}

\newcommand{\HC}{\mathbb{C}}

\newcommand{\HR}{\mathbb{R}}
\newcommand{\HS}{\mathbb{S}}

\newcommand{\kU}{\mathfrak{U}}

\newcommand{\cA}{\mathcal{A}}

\newcommand{\cC}{\mathcal{C}}

\newcommand{\cH}{\mathcal{H}}

\newcommand{\cL}{\mathcal{L}}

\newcommand{\cU}{\mathcal{U}}
\newcommand{\cV}{\mathcal{V}}
\newcommand{\cW}{\mathcal{W}}

\newcommand{\set}[1]{\left\{#1\right\}}
\newcommand{\cset}[2]{\left\{#1\mid #2\right\}}
\newcommand{\norm}[1]{\left\lvert #1 \right\rvert}

\newcommand{\rank}{\mathrm{rk}\,}

\renewcommand{\d}{\mathrm{d}}

\DeclareMathOperator*{\prob}{\mathrm{Prob}}
\DeclareMathOperator*{\Prob}{\mathrm{Prob}}
\DeclareMathOperator*{\mean}{\mathbb{E}}

\numberwithin{equation}{section}
\numberwithin{figure}{section}

\theoremstyle{plain}
\newcounter{numbering} \numberwithin{numbering}{section}
\newtheorem{thm}[numbering]{Theorem}
\renewtheorem{lemma}[numbering]{Lemma}
\newtheorem{prop}[numbering]{Proposition}
\newtheorem{cor}[numbering]{Corollary}
\theorembodyfont{\normalfont}
\newtheorem{dfn}[numbering]{Definition}
\newtheorem{rem}[numbering]{Remark}

\theoremstyle{nonumberplain}

\Crefname{equation}{}{}
\Crefname{equation}{}{}
\Crefname{theorem}{Theorem}{Theorems}
\begin{document}

\begin{center}
{\bf\Large \thetitle \\[1cm]}
Paul Breiding
\footnote{Institute of Mathematics, Technische Universität Berlin, breiding@math.tu-berlin.de. Partially supported by DFG research grant BU 1371/2-2.}
\hspace{1cm}
Peter Bürgisser
\footnote{Institute of Mathematics, Technische Universität Berlin, pbuerg@math.tu-berlin.de.\\Partially supported by DFG research grant BU 1371/2-2.}
\end{center}

\begin{abstract}
Let $f=(f_1,\ldots,f_n)$ be a system of $n$ complex homogeneous polynomials in $n$ variables of degree $d$. 
We call $\lambda\in\HC$ an eigenvalue of~$f$ if there exists
$v\in\HC^n\backslash\set{0}$ with $f(v)=\lambda v$, 
generalizing the case of eigenvalues of matrices ($d=1$). 
We derive the distribution of $\lambda$ when the $f_i$ are
independently chosen at random according to the unitary invariant Weyl
distribution and determine the limit distribution for $n\to\infty$.
\end{abstract}

\smallskip


\smallskip


\section{Introduction}

The theory of eigenvalues and eigenvectors of matrices is a well-studied subject in mathematics with a wide range of application. However, attempts to generalize this concept to homogeneous polynomial systems of higher degree have only been made very recently, motivated by tensor analysis \cite{qi2,qi1}, spectral hypergraph theory \cite{qi3} or optimization \cite{lim1}. An overview on recent publications can be found in \cite{lim2}, where the authors use the term "spectral theory of tensors".

Following Cartwright and Sturmfels, who in \cite{sturmfels-cartwright} adapt Qi's definition of E-eigenvalues, we say that a pair $(v,\lambda)\in(\HC^n\backslash\set{0})\times \HC$ is an \emph{eigenpair} of a system $f:=(f_1,\ldots,f_n)$ of $n$ complex homogeneous polynomials of degree~$d$ in the variables $X_1,\ldots,X_n$ if 
$f(v)=\lambda v$. 
We call $v$ an eigenvector and $\lambda$ an eigenvalue of $f$. If in addition $v^T\overline{v}=1$, we call the pair $(v,\lambda)$ normalized.

By \cite[Theorem 1.3]{lim3} we expect the task of computing eigenvalues of a given system to be hard. It is therefore natural to ask for the distribution of the eigenvalues, when the system $f$ is random.

In the case $d=1$ we obtain the  definition of eigenpairs of matrices. In \cite{ginibre} Ginibre assumes the entries of a complex matrix $A~=~(a_{i,j})$ to be independently distributed with density $\pi^{-1} \exp(-\lvert a_{i,j}\rvert^2)$ and describes the distribution of an eigenvalue~$\lambda$, 
that is chosen uniformly at random from the $n$ eigenvalues of $A$.

Can Ginibre's results be extended to arbitrary degree $d$? The answer is yes and provided in this paper. Let us call two eigenpairs $(v,\lambda),(w,\eta)$ \emph{equivalent} if there exists some~$t\in\HC\backslash\set{0}$, such that $(v,\lambda)=(tw,t^{d-1}\eta)$. 
Note that if both $(v,\lambda)$ and $(w,\eta)\in \cC$ are normalized, then we must have $\lvert t\rvert=1$.
This implies that the intersection of an equivalence class with the set of normalized eigenpairs of $f$ is a circle, that we assume to have volume $2\pi$. 
Cartwright and Sturmfels point out in \cite[Theorem 1]{sturmfels-cartwright} that if $d>1$, the number of equivalence classes of eigenpairs of a generic $f$ is $D(n,d):=(d^n-1)/(d-1)$. 

We define a probability distribution on the space of eigenvalues as follows:
\begin{enumerate}
\item \label{main_distribution1}
For each $1\leq i\leq n$ choose $f_i$ independently at random with the density $\pi^ {-k}\, \exp(-\lVert f_i\rVert^2)$, where $k:=\binom{n-1+d}{d}$. Here $\lVert \, \rVert$ is the unitary invariant norm on the space of homogeneous polynomials of degree $d$ defined in \Cref{se:geo-framework}, see also \cite[sec. 16.1]{condition}. (The resulting distribution of the $f_i$ is sometimes called the Weyl distribution.)
\item Among the $D(n,d)$ many equivalence classes $\cC$ of eigenpairs of $f$, 
choose one uniformly at random.
\item \label{main_distribution2} 
Choose an normalized eigenpair $(v,\lambda)\in\cC$ uniformly at random.
\item Apply the projection $(v,\lambda)\mapsto \lambda$.
\end{enumerate}

We denote by $\rho^{n,d}:\HC\to\HR_{\geq 0},\; \lambda\mapsto \rho^{n,d}(\lambda)$ 
the density of the resulting probability distribution. 
Observe that if $d=1$, then $\rho^{n,1}$ is the density of Ginibre's distribution. 

The unitary invariance of $\lVert \, \rVert^2$ implies that $\rho^{n,d}(\lambda)$ only depends on $\lvert \lambda\rvert$, but not on the argument of $\lambda$.
We therefore introduce the following notation:
	\begin{equation}\label{R}
	R:=2 \norm{\lambda}^2.
	\end{equation}
We will prove that the random variable $R$ follows a distribution that, if $d=1$, 
is mixed from $\chi^2$-distributions with weights from the uniform distribution on~$n$ items, 
and, if $d>1$, is mixed from $\chi^2$-distributions with weights from the geometric distribution $\mathrm{Geo}(p)$ truncated at $n$.
(See \Cref{truncated_dfn} for details on the truncated geometric distribution.)  

Here is our main result.

\begin{thm}\label{main_thm}
Let $n,d\geq 1$ and $\lambda$ be distributed with density $\rho^{n,d}$. Let $\rho_\HR^{n,d}$ denote the density of $R=2\lvert\lambda\rvert^2$.
\begin{enumerate}
\item If $d=1$, then \label{p1} 
	\begin{equation*}
	\rho^{n,1}_\HR(R) = \frac{1}{n}\, \sum_{k=1}^{n} \chi^2_{2k}(R)= \sum_{k=1}^{n} \prob\limits_{X\sim\mathrm{Unif}(\set{1,\ldots,n})}\set{X=k} \,  \chi^2_{2k}(R).
	\end{equation*}  
\item If $d>1$, then \label{p2} 
	\begin{align*}
	\rho^{n,d}_\HR(R) &= \frac{d-1}{d^n-1}\, \sum_{k=1}^{n} d^{n-k}  \,  \chi^2_{2k}(R)\\
	&= \sum_{k=1}^{n} \prob\limits_{X\sim \mathrm{Geo}(1-\frac{1}{d})}\cset{X=k}{X\leq n} \,  \chi^2_{2k}(R).
	\end{align*}  
\end{enumerate}
Here $\chi^2_{2k}(R):=(e^{-\frac{R}{2}} R^{k-1})/(2^k (k-1)!)$ is the the density of a chi-square distributed random variable with $2k$ degrees of freedom.
\end{thm}

We note that 
$\prob\limits_{X\sim \mathrm{Geo}(p)}\set{X=k}$
is the probability that the first success of independent Bernoulli trials, each with success probability $p$, is achieved in the $k$-th trial. 
Moreover, for $1\leq k\leq n$ and $0\leq q < 1$ we have
	\begin{align*}
	\sum_{t=0}^\infty \prob\limits_{X\sim \mathrm{Geo}(1-q)}\set{X=k+tn} &= q^{k-1}(1-q)\sum_{t=0}^\infty q^{tn} \\
               &= \prob\limits_{X\sim \mathrm{Geo}(1-q)}\cset{X=k}{X\leq n};\end{align*}
for the last equality see \Cref{truncated_dfn}. 
One can therefore sample $\lvert \lambda\rvert^2$ by the following procedure.

\begin{enumerate}
\item If $d=1$, choose $k\in\set{1,\ldots,n}$ uniformly at random.
\item If $d>1$, make Bernoulli trials with success probability $1-\frac{1}{d}$ until the first success. 
Let $\ell$ be the number of the last trial and $k$ the remainder of~$\ell$ when divided by $n$.
\item Choose $x_1,\ldots,x_{2k} \stackrel{iid}\sim N(0,1)$.
\item Put $R:=\sum\limits_{i=1}^{2k} x_i^2$.
\item Output: $\frac{1}{2}R$.
\end{enumerate}

\begin{rem}\label{connection}
By de l'Hopital's rule  we have $\lim\limits_{d\to 1}
\frac{d-1}{d^n-1} = \frac{1}{n}.$ This implies that $\lim\limits_{d\to 1} \rho^{n,d}(R)=\rho^{n,1}(R)$, 
which yields a connection between the cases $d=1$ and $d>1$ 
(observe that here we allowed $d$ to be any real number). 
\end{rem}
We can compute the expectation of the random variable $\lvert \lambda\rvert^2$; cf.~Figure~\ref{fig1}. 
\begin{cor}\label{second_main_result}
If $d=1$, then 
$\mean\limits_{\lambda\sim \rho^{n,1}} \lvert \lambda\rvert^2 =\frac{n+1}{2}$.
If $d>1$, then 
\[\mean\limits_{\lambda\sim \rho^{n,d}} \lvert\lambda\rvert^2 \ =\  \frac{n -(n+1)d +d^{n+1}}{(d^n-1)(d-1)}.\]
We have 
$\lim\limits_{d\to \infty} \mean\limits_{\lambda\sim  \rho^{n,d}} \lvert\lambda\rvert^2 =1$ 
and
$\lim\limits_{n\to \infty} \mean\limits_{\lambda\sim \rho^{n,d}} \lvert\lambda\rvert^2 \ =\  \frac{d}{d-1}$ 
if $d>1$.
Moreover, for fixed~$n$, the function $d\mapsto \mean\limits_{\lambda\sim
  \rho^{n,d}} \lvert\lambda\rvert^2$ is strictly decreasing. 
For fixed~$d$, the function $n\mapsto \mean\limits_{\lambda\sim \rho^{n,d}} \lvert\lambda\rvert^2 $ is strictly increasing. 
\end{cor}
\begin{figure}[h] 
	\centerline{\begin{minipage}{0.5\textwidth}
  \includegraphics[width=\textwidth]{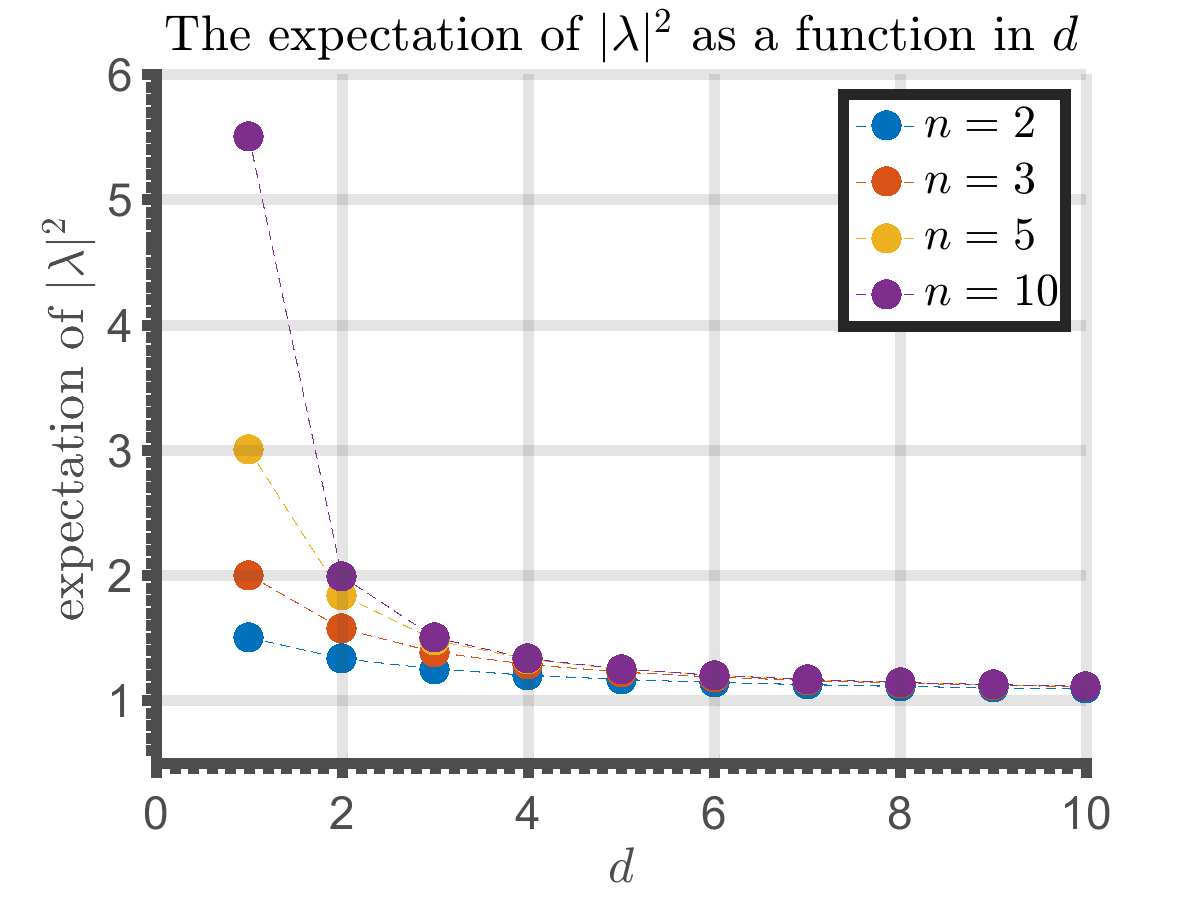}
  \end{minipage}\hfill
  \begin{minipage}{0.5\textwidth}
  \includegraphics[width=\textwidth]{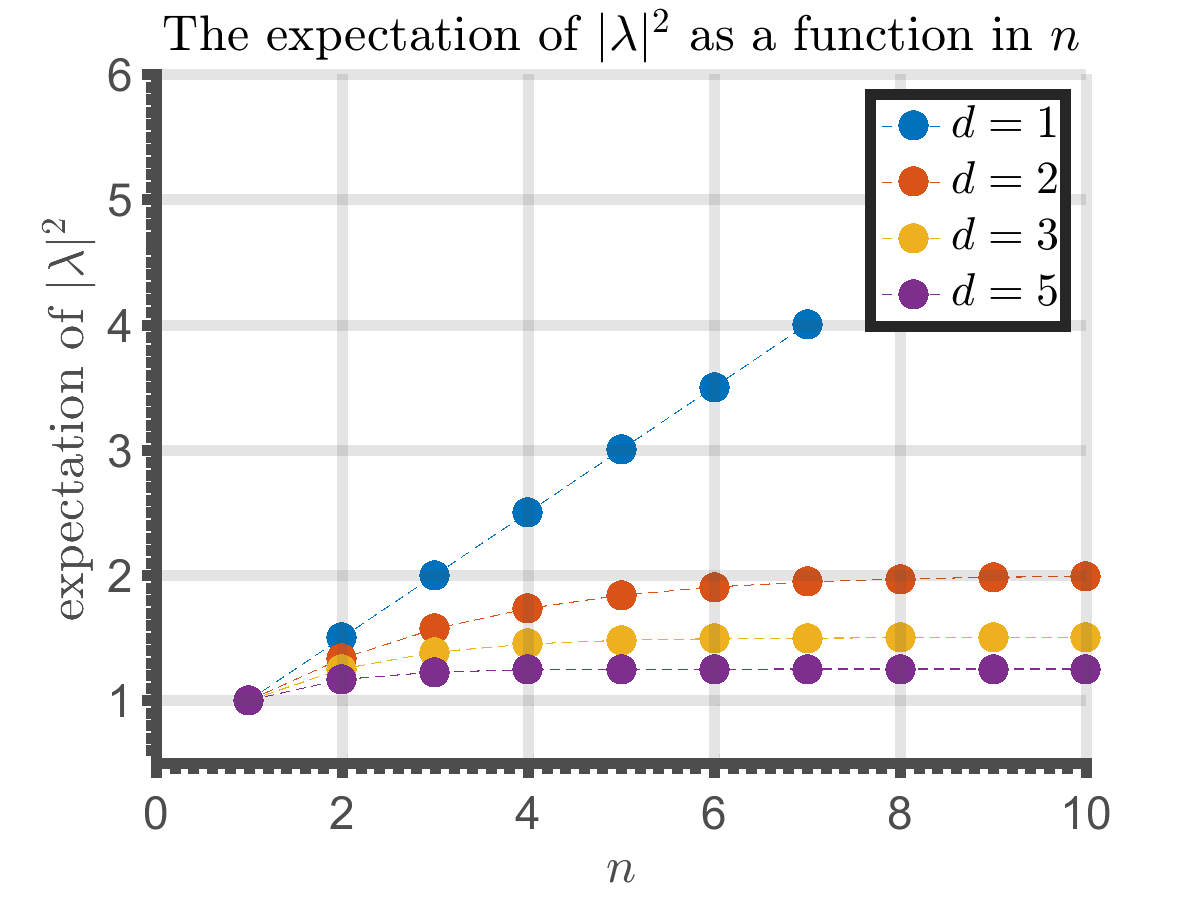}
  \end{minipage}
  }
	\caption{\small The left picture shows plots of $d\mapsto\mean\limits_{\lambda\sim\rho^{n,d}}\lvert\lambda\rvert^2$ for $n\in\{2,3,5,10\}$. 
	On the right are plots of $n\mapsto\mean\limits_{\lambda\sim\rho^{n,d}}\lvert\lambda\rvert^2$ for $d\in\{1,2,3,5\}$.}
	\label{fig1}
\end{figure}
In order to investigate $\rho_\HR^{n,1}$ for large $n$, we can normalize $\lvert\lambda\rvert^2$ by dividing it by its expectation. 
We will, however, divide $\lvert\lambda\rvert^2$ by $n$. While for large~$n$ this does not make a big difference, the formulas appearing are easier to understand. 
We will also normalize in the case $d>1$. So if $d=1$, we put
	\[\tau:= \frac{\lvert\lambda\rvert^2}{n}=\frac{R}{2n},\]
and if $d>1$, we put
	\begin{equation}\label{tau}
	\tau:= \frac{\lvert\lambda\rvert^2}{2\lim\limits_{n\to\infty} \mean\limits_{\lambda\sim \rho^{n,d}}\lvert\lambda\rvert^2} = \frac{R(d-1)}{4d}.
	\end{equation}
Making a change of variables from $R$ to $\tau$ yields the \emph{normalized density}, denoted by $\rho^{n,d}_\textbf{norm}$. 
In \cite{ginibre} Ginibre notes that in the case $d=1$ we have
	\begin{equation}
	\lim\limits_{n\to\infty}\rho^{n,1}_\textbf{norm}(\tau) =  \mathbf{1}_{[0,1]}(\tau):=
        \begin{cases} 1, & \text{ if } 0\leq \tau\leq 1\\ 0, & \text{ else}\end{cases}\label{ginibre}
	\end{equation}
This means that the distribution of the normalized eigenvalue $\lambda/\sqrt{n}$ converges towards the uniform distribution on 
the unit ball $\cset{x\in\HC}{\lvert x\rvert \leq 1}$. 

Our third result 
covers the case $d>1$.

\begin{thm}\label{third_main_result}
Let $d>1$ be fixed. For any $\tau\geq 0$ we have
	\[\lim_{n\to \infty} \rho^{n,d}_\textbf{norm}(\tau) =2\, e^{-2\tau}.\]
Hence, as $n\to\infty$, the normalized density $ \rho^{n,d}_\textbf{norm}(\tau)$ converges towards the exponential distribution 
with parameter $2$ (cf.~Figure~\ref{fig6}).
\end{thm}
\begin{figure}[h]
	\centerline{
	\begin{minipage}{0.5\textwidth}
  \includegraphics[width=\textwidth]{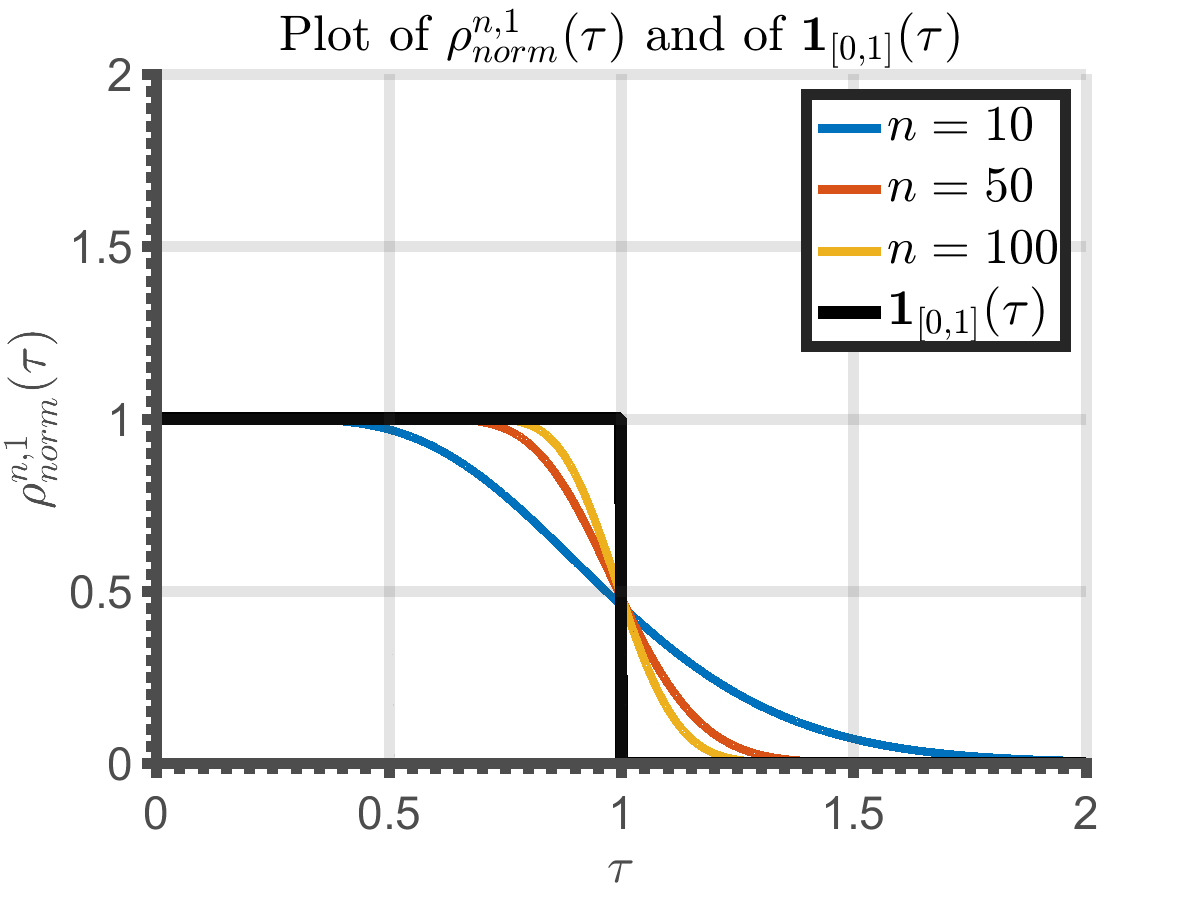}
  \end{minipage}\hfill
  \begin{minipage}{0.5\textwidth}
  \includegraphics[width=\textwidth]{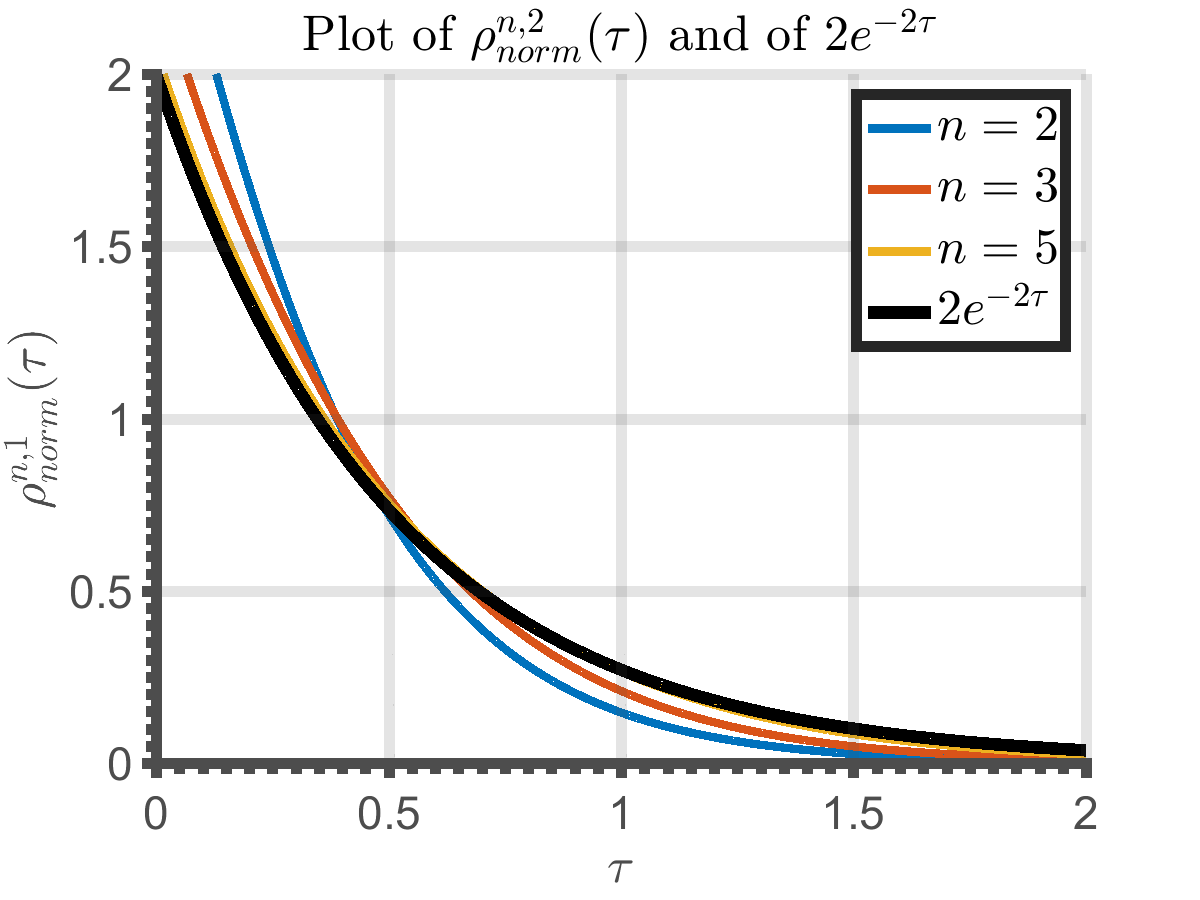}
  \end{minipage}
}
	\caption{\small The picture on the left shows plots of $\rho^{n,1}_\textbf{norm}(\tau)$ for $n\in\{10,50,100\}$ together with $\mathbf{1}_{[0,1]}(\tau)$. On the right are plots of $\rho^{n,2}_\textbf{norm}(\tau)$ for $n\in\{2,3,5\}$ together with $2e^{-2\tau}$.}
	\label{fig6}
\end{figure}

\subsection{Relation to prior work}

Our definition of eigenpairs is inspired by the following definition of E-eigenvalues of tensors and supermatrices 
given by Qi in \cite[sec.~2-3]{qi1} and \cite[sec.~1]{qi2}. 

Let $\Phi$ be a multilinear map $(\HR^n)^ d \to \HR^ n$, that is represented by the real supermatrix $A=(A_{j,i_1\ldots i_d})$. For $v\in\HR^n$, 
Qi puts $A\, v^d:= A(v,\ldots,v)$ (see \cite[sec.~3, eq.~(7)]{qi1}) and then defines $\lambda\in\HC$ to be an E-eigenvalue of $A$ 
if there exists $v\in \HC^n$ such that
$A \, v^d = \lambda v$ and $v^Tv=1$.
This definition of eigenvalue is independent of the change of orthonormal coordinates. Therefore, $\lambda$ can be regarded as an eigenvalue of $\Phi$ itself. Although assuming $\Phi$ over the reals, Qi allows the eigenvalue to be complex. If the eigenvalue $\lambda$ is real, he calls it a Z-eigenvalue. 

In \cite{sturmfels-cartwright} Cartwright and Sturmfels relax the  definition of Qi by considering order-$(d+1)$ tensors/supermatrices over the complex numbers $\HC$. They define a pair $(v,\lambda)\in (\HC^n\backslash\set{0})\times~\HC$ to be an eigenpair of~$A$, if
	 \begin{equation}\label{eigenvalue_supermatrix}
	  A \, v^d = \lambda v.
	 \end{equation}
Observe that Qi's condition $v^Tv=1$ implies that $v\neq 0$, while Sturmfels and Cartwright require the eigenpair to be an element in $(\HC\backslash\set{0})\times \HC$. In reference to Qi they call an eigenpair $(v,\lambda)$ satisfying $v^Tv=1$ normalized, whereas we call an eigenpair $(v,\lambda)$ normalized, if it satisfies $v^T\overline{v}=1$. In fact,~$A \, v^d$ is a system of homogeneous polynomials over $\HC$ in the entries of~$v$. So \Cref{eigenvalue_supermatrix} coincides with our definition.

Another approach to define eigenpairs of homogeneous polynomial systems is given by Lim \cite{lim1} in his variational approach, which is as follows. 
 
We denote by $\lVert \, \rVert_k$ the $\cL_k$-norm on $\HR^n$ for $k>1$.   
Suppose that~$F(X)$ is a real homogeneous polynomial in $n$ variables $X=(X_1,\ldots,X_n)$ of degree $d+1$. In order to optimize $F$ on the $\cL_k$-sphere $\set{\lVert x\rVert_k=1}$, one can consider the Langrangian of the multilinear Rayleigh quotient $F(X)/\lVert X\rVert_k^{d+1}$, that is $L(X,\Lambda):=F(X)-(d+1)^ {-1}\Lambda\, (\lVert X\rVert_k^{d+1}-1),$
where $\Lambda$ is an auxiliary variable. Then the equation $\nabla L=0$ gives
	\begin{equation}\label{lim}
	\nabla F(X) = \Lambda \, \begin{pmatrix} \mathrm{sgn}(X_1)^k\, X_1^{k-1}\\\vdots\\\mathrm{sgn}(X_n)^k\, X_n^{k-1} \end{pmatrix},\quad  \lVert X\rVert_k=1.
	\end{equation}
Note that $\nabla F(X)$ is a system of homogeneous polynomials of degree $d$. If the pair $(v,\lambda)\in \{\lVert v\rVert_k=1\}\times \HR$ is a solution of equation \Cref{lim}, Lim calls $v$ an $\cL^k$-eigenvector and $\lambda$ an $\cL^k$-eigenvalue of the system $\nabla F$. In particular, if~$k=2$, the $\cL_2$-eigenvalues $(v,\lambda)\in \{\lVert v\rVert_2=1\}\times \HR$ satisfy 
	\begin{equation*}
	\nabla F(v) = \lambda \, v,\quad  \lVert v\rVert_2=1.
	\end{equation*}
If we relax the definition of $\cL_2$-eigenvalues by allowing $(v,\lambda)$ to be complex, the pair $(v,\lambda)$ is an eigenpair of the system $\nabla F(X)$ 
in our sense.

\medskip

The organization of the paper is as follows. 
After some preliminaries presented in the next section, we establish in Section~\ref{se:geo-framework}
the geometric framework for the eigenpair problem. 
Our concepts and notations are close to the ones from \cite[sec. 16]{condition}. 
We define a probability distribution on $\cset{(f,v,\lambda)}{f(v)=\lambda v}$, the solution manifold.
The pushforward measure of this distribution with respect to the projection onto the space of eigenvalues 
is precisely $\rho^{n,d}$. Finally, we prove the stated results in Section~\ref{se:proofs}.  

\medskip
\noindent{\bf Acknowledgements}
The basis of this work was laid during the program "Algorithms and Complexity in Algebraic Geometry" at the Simons Institute for the Theory of Computing. We are grateful for the Simons Institute for the stimulating environment and the financial support.
\section{Preliminaries}
\subsection{Differential geometry}
We denote by $\langle x,y\rangle := x^T\overline{y}$ the standard hermitian inner product on $\HC^{n}$. 
Furthermore, we set $\lVert x\rVert := \sqrt{\langle x,x\rangle}$ and 
$\HS(\HC^n):=\big\{x\in\HC^n \mid \lVert x\rVert =1\big\}$. 
Given some $x\in\HC^n\backslash\set{0}$ we denote by
$T_x := \big\{y\in \HC^n \mid \langle x,y\rangle =0 \big\}$ 
the orthogonal complement of $x$ in $\HC^n$.

If $M$ is a differentiable manifold and $x\in M$ we denote by
$T_x M$ the tangent space of $M$ at $x$.
\begin{lemma}\label{tangentspace_sphere}
Let $v\in\HS(\HC^n)$. Then 
$T_v\HS(\HC^n)=\big\{a\in\HC^n \mid \Re \langle a,v\rangle =0\big\} = T_v \oplus \HR iv$ 
and this composition is orthogonal with respect to inner product on
$T_v\HS(\HC^n)$, that is induced from $\langle\: ,\rangle$. 
\end{lemma}
\begin{proof}
See \cite[Equation (14.11)]{condition} and \cite[Lemma 14.9]{condition}.
\end{proof}
If $M$ and $N$ are differentiable manifolds and $F\colon M\to N$ is
differentiable, we denote by $DF(x):T_x M \to T_{F(x)}N$ its derivative at $x\in M$ and by $NJ(F)(x)$ its normal jacobian at $x$.

For more details on normal jacobians and the coarea formula we refer to \cite[sec. 17.3]{condition}.
\begin{thm}[Coarea formula]\label{coarea_formula}
Suppose that $M,N$ are Riemannian manifolds of dimensions $m,n$, respectively. Let $\Psi: M\to N$ be a surjective smooth map. Then we have for any function $\chi: M\to \HR$ that is integrable with respect to the volume measure of $M$ that
	\[\int_M \chi \d M = \int_{y\in N} \left[ \int_{\Psi^{-1}(y)} \frac{\chi}{NJ(\Psi)} \d \Psi^{-1}(y)\right] \d N.\]
\end{thm}


Let $E_1,E_2$ be finite dimensional complex vector spaces with hermitian inner product, such that $\dim_\HC E_1 \geq \dim_\HC E_2$. Assume that we have a surjective linear map $\phi:E_1 \to E_2$ (think of $\phi$ as a derivative and $E_1,E_2$ being tangent spaces).  Let $\Gamma(\phi):=\set{(x,\phi(x))\in E_1\times E_2}$ be the graph of $\phi$. Then $\Gamma(\phi)$ is a linear space and the projections $p_1: \Gamma(\phi)\to E_1$ and $p_2: \Gamma(\phi)\to E_2$ are linear maps.

The following result is Lemma 3 in \cite[sec. 13.2]{BSS}, combined with the comment in Theorem 5 in \cite[sec. 13.2]{BSS}. 
\begin{lemma}\label{NJ_smale_lemma}
Let $W$ be the orthogonal complement of $\ker p_2$. Then we have \[\frac{\lvert\det (p_1)\rvert}{\lvert \det (p_2|_W)\rvert}= \lvert\det(\phi\phi^*)\rvert^{-1}.\]
\end{lemma}

\subsection{Expectation of the truncated geometric distribution}

The \emph{geometric distribution with parameter $p$ truncated at} $n\geq 1$ is defined to be the distribution 
of a geometrically distributed random variable $X$ with parameter $p$ under the condition that $X\leq n$. Its density is
	\begin{equation}\label{truncated_dfn}
	\Prob\limits_{X\sim\mathrm{Geo}(p)}\cset{X=k}{X\leq n} =
        \frac{\Prob\limits_{X\sim\mathrm{Geo}(p)}\set{X=k}}{\Prob\limits_{X\sim\mathrm{Geo}(p)}\set{X\leq  n}}  
          = \frac{q^{k-1}(1-q)}{1-q^n},
	\end{equation}
where $q:=1-p$ and $k\in\set{1,\ldots,n}$.
\begin{lemma}\label{expectation_truncatedgeo}
Let $n\geq 1$ and $0\leq q< 1$. Then
	\[\mean\limits_{X\sim \mathrm{Geo}(1-q)}[X\mid X\leq n] = \frac{nq^{n+1} -(n+1)q^n +1}{(1-q^n)(1-q)}.\]
\end{lemma}

\begin{proof}
We have $\Prob\limits_{X\sim \mathrm{Geo}(1-q)}\cset{X = k}{X\leq n} = \frac{q^{k-1}(1-q)}{1-q^n}$, $k\in\set{1,\ldots,n}$. This implies \[\mean\limits_{X\sim \mathrm{Geo}(1-q)}[X \mid X\leq n]=\frac{1-q}{1-q^n}\, \sum_{k=1}^{n} kq^{k-1}.\] Observe, that $\sum_{k=1}^{n} kq^{k-1}$ is the derivative of $\frac{1-Z^{n+1}}{1-Z}$ at $Z=q$ and that
\begin{align*}
	\frac{\d\;}{\d Z} \left(\frac{1-Z^{n+1}}{1-Z}\right)=\frac{nZ^{n+1} - (n+1)Z^n +1 }{(1-Z)^2}	
\end{align*}
Hence the claim.
\end{proof}

\subsection{The expected characteristic polynomial of a random matrix}\label{sec_expected}

We say that a random variable $z$ on $\HC$ is \emph{standard normal distributed} if both real and imaginary part of $z$ 
are i.i.d centered normal distributed random variables with variance $\sigma^2=\frac{1}{2}$. The corresponding density is
	\begin{align*}
	\varphi(z) := \frac{1}{\pi}\, \exp\left(-\lvert z\rvert^2\right),
	\end{align*}
and we write $z\sim N(0,\frac{1}{2})$ for this distribution. The
reason why we have put $\sigma^2=\frac{1}{2}$ is that 
for a gaussian random variable $z\sim N(0,\frac{1}{2})$ on $\HC$ we have
	\begin{equation}\label{chi_square}\mean\limits_{z\sim N(0,\frac{1}{2})}\lvert z\rvert^2 = 1.
	\end{equation}

Suppose that $E$ is a finite dimensional complex vector space with hermitian inner product and let $k:=\dim_\HC E$. 
We define the standard normal density on the space $E$ as 	
	\begin{equation}\label{varphi_E}
	\varphi_E(z) :=   \frac{1}{\pi^k}\, \exp\left(-\lVert z\rVert^2\right).
	\end{equation}
 it is clear from the context which space is meant, we omit the subscript~$E$. Let $I_{n}$ be the $n\times n$ identity matrix. 
If a complex matrix $A\in\HC^{n\times n}$ is distributed with density $\varphi_{\HC^{n\times n}}$, we write $A\sim N(0,\frac{1}{2} I_{n})$. 

Recall that the Gamma function is defined by 
$\Gamma(n):= \int_{t=0}^\infty t^{n-1} e^{-t} \d t$ for a positive
real number $n >0$. It is well known that $\Gamma(n)=(n-1)!$ if
$n$ is a positive integer. 
The \emph{upper incomplete Gamma function} is defined as 
\[\Gamma(n,x):= \int_{t=x}^\infty t^{n-1} e^{-t} \d t,\]where $x\geq 0$. 
 
\begin{lemma}\label{auxiliary_lemma}~
\begin{enumerate}
\item 
Let $x\geq 0$ and $n\geq 1$. Then $\Gamma(n,x) = (n-1)! \, e^{-x} \sum\limits_{k=0}^{n-1} \frac{x^k}{k!}$.
\item We have $\mean\limits_{A\sim N\left(0,\frac{1}{2}I_{n}\right)} \lvert \det(A)\rvert^2 = n! = \Gamma(n-1).$
\item  For $I\subset[n]:=\set{1,\ldots,n}$ we define $A_I\in\HC^{\lvert I\rvert \times \lvert I\rvert}$ to be the submatrix of $A\in\HC^{n\times n}$
indexed by $I$. Then for any $t\in \HC$  we have that $\det\left(A+tI_{n}\right) =  \sum_{I\subset[n]} t^{n-\lvert I\rvert} \det A_I.$
	\end{enumerate}
\end{lemma}

\begin{proof}
The first assertion is from \cite[p. 949]{gradshteyn}, 
the second is \cite[Lemma 4.12]{condition}, 
and the third assertion is a well known fact, cf.~\cite[Theorem 1.2.12]{johnson}.
\end{proof}

\begin{prop}\label{expectation_det}
We have for $A\in\HC^{n\times n}$ and $t\in \HC$ 
\[\mean\limits_{A\sim N\left(0,\frac{1}{2}I_{n}\right)} \lvert \det (A+tI_{n})\rvert^2 =  e^{\lvert t\rvert^2} \Gamma\left(n+1,\lvert t\rvert^2\right).\]
\end{prop}
\begin{proof}
By \Cref{auxiliary_lemma}(3), $\det (A+tI_{n\times n}) = \sum_{\alpha \in \set{0,1}^n} t^{n-\lvert \alpha \rvert} \det A_\alpha$, 
hence
	\[\lvert\det (A+tI_{n\times n})\rvert^2 = \sum_{\alpha,\beta} t^{n-\lvert \alpha \rvert} \, (\overline{t})^{n-\lvert \beta \rvert} \det A_\alpha \det \overline{A_\beta}.\]
Due to \Cref{auxiliary_lemma}(2), we have 
$\mean\, [\det A_\alpha \det \overline{A_\beta}] = \delta_{\alpha,\beta}\, \lvert\alpha\rvert!$, 
since we deal with centered distributions. Hence, 
\[\mean\, \lvert\det (A+tI_{n\times n})\rvert^2 = \sum_{k=0}^n \binom{n}{k}\, k!\,\lvert t\rvert^{2(n-k)} 
    = e^{\lvert t\rvert^{2}}\,\Gamma(n+1,\lvert t\rvert^{2});\]
the last equality by \Cref{auxiliary_lemma}(1).
\end{proof}


\section{Geometric framework}\label{se:geo-framework}

\subsection{Eigenpairs of homogeneous polynomial systems}


Let $n,d\geq 1$. 
We denote by $\cH_d:=\cH_d(X_1,\ldots,X_n)$ the vector space of
homogeneous polynomials of degree~$d$ in the variables
$X_1,\ldots,X_n$ over the complex numbers $\HC$ of degree $d$. 
The \emph{Bombieri-Weyl} basis is given by the~$e_\alpha:=\binom{d}{\alpha}^\frac{1}{2} X^\alpha$, $\lvert \alpha \rvert = d$. We define an inner product on $\cH_d$ via
	\begin{equation}\label{inner_product}
	\Big\langle \sum\limits_\alpha a_\alpha e_\alpha,\sum\limits_\alpha b_\alpha e_\alpha \Big\rangle := \sum\limits_\alpha a_\alpha \overline{b_\alpha}.
	\end{equation}
The product \Cref{inner_product} extends to $(\cH_d)^n$ in the following way. Let $f=(f_1,\ldots,f_n)$ and $g=(g_1,\ldots,g_n)\in(\cH_d)^n$. Then we define $\langle f,g\rangle :=\sum_{i=1}^n \langle f_i,g_i\rangle.$ Moreover, for $f\in(\cH_d)^n$ we set $\lVert f\rVert := \sqrt{\langle f,f\rangle}$. 
\begin{rem}\label{f_norm_is_frobenius}
\begin{enumerate}
\item 
The inner product \Cref{inner_product} is the unique unitary invariant product on $\cH_d$ (up to scaling). See \cite[Theorem 16.3]{condition} and \cite[Remark 16.4]{condition}.
\item Suppose that $f=(f_1,\ldots,f_n)\in(\cH_d)^n$ and $f_i=\sum_{\alpha} a_{i,\alpha} e_\alpha$, $1\leq i\leq n$. Let $k:=\dim \cH_d$ and put $A:=(a_{i,\alpha})\in\HC^{n\times k}$. Then $\lVert f\rVert = \lVert A\rVert_F$, where $\lVert \, \rVert_F$ is the Frobenius norm.
\end{enumerate}
\end{rem}

For the sake of clarity, we recall the definition of eigenpairs given in the introduction.
\begin{dfn}
An \emph{eigenpair} of $f\in (\cH_d)^n$ is a pair
$(v,\lambda)\in(\HC^n\backslash\{0\}) \times \HC$ such that 
$f(v) = \lambda v$. We call $v$ an \emph{eigenvector} and $\lambda$ an \emph{eigenvalue} of $f$.
Further, we call eigenpairs $(v,\lambda)$ and  $(w,\eta)$ \emph{equivalent}, 
$(v,\lambda) \sim (w,\eta)$, if there exists a nonzero $t\in\HC$ such that 
$(tv,t^{d-1}\lambda) = (w,\eta)$.
\end{dfn}

We already noted that the number of equivalence classes of a generic system $f$ equals
$D(n,d) = (d^n-1)/(d-1)$ if $d>1$, cf.~\cite{sturmfels-cartwright}. 

\subsection{The solution manifold}

Let $\cA:=\HC[X_1,\ldots,X_n,\Lambda]$ be the space of polynomials in the $n+1$ variables $X_1,\ldots,X_n,\Lambda$. 
We consider the map $F: (\cH_d)^n \to \cA^n,f\mapsto f(X)-\Lambda X$.  
For~$f\in (\cH_d)^n$ we set $F_f:=F(f)$, such that 
\begin{equation}\label{dfn_F_A}
F_f:\HC^n\times \HC \to \HC^n,\quad (v,\lambda) \mapsto f(v)-\lambda v.
\end{equation}
Observe that $F_f(X,\Lambda)$ consists of two parts, one homogeneous of degree $d$ and one homogeneous of degree $2$. 
Let us denote by $\partial_X$ and $\partial_\Lambda$ the partial derivatives of $F_f(X,\Lambda)$ with respect to $X=(X_1,\ldots,X_n)$ and $\Lambda$, respectively. 
Then the derivative of $F_f$ at $(v,\lambda)$ has the following matrix representation:
\begin{equation}
	\begin{bmatrix} \partial_X f -\partial_X (\Lambda X), &  -\partial_\Lambda(\Lambda X)\end{bmatrix}_{(X,\Lambda)=(v,\lambda)}
	= \begin{bmatrix} \partial_X f(v) -\lambda I_{n}, &  -v\end{bmatrix}\label{jacobian},
\end{equation} 
where $I_{n}$ denotes the $n\times n$-identity matrix. 

We adapt the terms ``solution manifold'' and ``well-posed''  
from \cite[sec.~16.2]{condition} and tailor them to our (structured) set $\cset{F_f}{f\in(\cH_d)^n}$. 
Compare \cite[Open Problem 15]{condition}. We call 
\begin{equation*}\label{sol_variety_sphere}
	\cV :=\cset{(f,v,\lambda)\in(\cH_d)^n\times \HS(\HC^n)\times\HC}{F_f(v,\lambda)=0},
\end{equation*}
the \emph{solution manifold} and its subset 
\begin{equation*}
        \cW :=\cset{(f,v,\lambda)\in \cV}{\rank DF_f(v,\lambda)=n}
\end{equation*}
the \emph{manifold of well-posed} triples.

The group $\cU(n)$ of unitary linear transformations $\HC^{n}\to\HC^{n}$ acts on 
$\left(\cH_d\right)^n$ and $\cV$, respectively, via  
\begin{equation}\label{groupaction}
  U.f:=U\circ f\circ U^{-1}\quad\text{and}\quad U.(f,v,\lambda) := (U.f, Uv, \lambda) .
\end{equation}
We note that 
$\cW$ is invariant under the group action and that $\cU(n)$ acts by isometries;  
see \cite[Theorem 16.3]{condition}.

\begin{lemma}\label{tangent_space_sphere} 
The solution manifold $\cV$ is a connected and smooth submanifold  of $(\cH_d)^n\times \HS(\HC^n)\times \HC$ of 
dimension $\dim_\HR \cV=\dim_\HR (\cH_d)^n+1$.
Moreover, the tangent space of $\cV$ at $(f,v,\lambda)$ equals \[\cset{(\dot{f},\dot{v},\dot{\lambda})\in (\cH_d)^n\times T_v\HS(\HC^n)\times \HC}{ \dot{f}(v) + DF_f(v,\lambda)(\dot{v},\dot{\lambda})=0}.\]
\end{lemma}

\begin{proof}
The map $G: (\cH_d)^n\times \HS(\HC^n)\times \HC \to \HC^n, (f,v,\lambda)\mapsto F_f(v,\lambda)$ 
has $\cV$ as its fiber over $0$. 
The derivative of $G$, 
$$
	DG(f,v,\lambda):\; (\cH_d)^n\times T_v\HS(\HC^n)\times \HC \to \HC^n,\  
	(\dot{f},\dot{v},\dot{\lambda}) \mapsto \dot{f}(v) + DF_f(v,\lambda)(\dot{v},\dot{\lambda}) , 
$$
is clearly surjective. 
Therefore $0\in\HC^n$ is a regular value of $G$ and 
Theorem~A.9 in \cite{condition} implies the assertion.
\end{proof}

The following lemma is easily verfied using Euler's identity for homogeneous functions.
\begin{lemma}\label{kernel_D}
Let $(f,v,\lambda)\in \cW$. Then $\ker DF_f(v,\lambda)=\HC\,(v,(d-1)\lambda)^T$. In particular, $DF_f(v,\lambda)|_{T_v\times \HC}$ is invertible.
\end{lemma}

\begin{cor}\label{cor_tangent_space}
The tangent space $T_{(f,v,\lambda)}\cV$ at $(f,v,\lambda)\in \cW$ is given by 
$$
\Big\{(\dot{f},\dot{v}+riv,\dot{\lambda})\in (\cH_d)^n\times (T_v \oplus \HR i v)\times \HC \mid (\dot{v},\dot{\lambda}) 
 = - DF_f(v,\lambda)|_{T_v\times \HC}^{-1}\, \dot{f}(v) \Big\}.
$$
\end{cor}
\begin{proof}
Let $(f,v,\lambda)\in\cV$ be fixed. By \Cref{tangent_space_sphere} the tangent space of $\cV$ at $(f,v,\lambda)$ equals 
	\begin{equation*}
\cset{(\dot{f},\dot{w},\dot{\lambda})\in (\cH_d)^n\times T_v\HS(\HC^n)\times \HC}{ DF_f(v,\lambda)\, (\dot{w},\dot{\lambda})= -\dot{f}(v)}.
\end{equation*}
From \Cref{tangentspace_sphere} we know that $T_v \HS(\HC^n) = T_v \oplus \HR i v$. \Cref{kernel_D} tells us that 
	\[\ker DF_f(v,\lambda) = \HC (v,(d-1)\lambda) = \HR (v,(d-1)\lambda) \oplus \HR i (v,(d-1)\lambda).\]
Hence, $(T_v \HS(\HC^n) \times \HC) \cap \ker DF_f(v,\lambda) = \HR i (v,(d-1)\lambda)$. 
From \Cref{kernel_D} we know that $DF_f(v,\lambda)|_{T_v\times\HC}$ is invertible. 
We conclude that if $(\dot{f},\dot{w},\dot{\lambda})\in T_{(f,v,\lambda)}\cV$, then there exist uniquely determined 
$\dot{v}\in T_v$ and $r\in \HR$ such that $\dot{w}=\dot{v}+irv$ and $DF_f(v,\lambda)(\dot{w},\dot{\lambda}) = DF_f(v,\lambda)(\dot{v},\dot{\lambda})$, from which the claim follows.
\end{proof}

\subsection{Projections and normal jacobians}

We consider the projections 
\begin{equation}\label{projections1}
\pi_1\colon\cV\to (\cH_d)^n, (f,v,\lambda)\mapsto f,\quad 
\pi_2\colon\cV\to \HS(\HC^n)\times \HC, (f,v,\lambda)\mapsto (v,\lambda) .
\end{equation}

It is essential that the quotient of the normal jacobians of $\pi_1$ and $\pi_2$ can be computed in the following way.

\begin{lemma}\label{quotient_NJ}
For all $(f,v,\lambda)\in \cW$ we have
	\[\frac{NJ(\pi_1)(f,v,\lambda)}{NJ(\pi_2)(f,v,\lambda)} = \left\lvert \det (DF_f(f,v,\lambda)|_{T_v\times \HC})\right\rvert^2.\]
\end{lemma}
\begin{proof}
Let $\cU(n)$ be group of unitary maps $\HC^n\to \HC^n$. Recall from \Cref{groupaction} that for $U\in\cU(n)$ and $(f,v,\lambda)\in\cV$ we have put $U.(f,v,\lambda):=(U.f,Uv,\lambda)$. 
By definition, the projections $\pi_1,\pi_2$ are $\cU(n)$-equivariant. 
Hence for any $U\in\cU(n)$ we have~$NJ(\pi_i)(f,v,\lambda) = NJ(\pi_i)(U.(f,v,\lambda))$, $i=1,2$. It therefore suffices to show the claim for $v=e_1$, where  $e_1:=(1,0,\ldots,0)\in\HC^n$.

Suppose that $(f,e_1,\lambda)\in \cW$. The derivatives of $\pi_1$ and $\pi_2$ are the projections
	\[
	\begin{array}{lll}
	D\pi_1(f,e_1,\lambda):& T_{(f,e_1,\lambda)} \cV \to (\cH_d)^n,& (\dot{f},\dot{v},\dot{\lambda})\mapsto \dot{f},\\
	D\pi_2(f,e_1,\lambda):& T_{(f,e_1,\eta)} \cV \to T_{e_1}\HS(\HC^n)\times \HC,& (\dot{f},\dot{v},\dot{\lambda})\mapsto (\dot{v},\dot{\lambda}).
	\end{array}\]
Let us write $\dot{f} = \sum\limits_{\alpha} \dot{f}_\alpha X^\alpha$, where for all $\alpha$ we have $\dot{f}_\alpha\in\HC^n$. Then we obtain $\dot{f}(e_1) = \dot{f}_{(d,0,\ldots,0)}$. Hence, $\dot{f}\mapsto \dot{f}(e_1)$ is an orthogonal projection. We will denote this projection by $\Pi$. By \Cref{tangentspace_sphere} the projection $\dot{v}+riv\mapsto \dot{v}$ is orthogonal as well. Using  \Cref{cor_tangent_space} it follows that 
$T_{(f,e_1,\lambda)}\cV$ is the graph of the surjective linear function 	\[-DF_f(e_1,\lambda)|_{T_{e_1}\times \HC}^{-1} \;\circ\; (\Pi\times 0):\: (\cH_d)^n\times \HR iv \to T_{e_1}\times \HC.\]
Applying \Cref{NJ_smale_lemma} yields the claim.
\end{proof}

\subsection{The eigendiscriminant variety}

We define the set of \emph{ill-posed} triples $(f,v,\lambda)$ to be
\begin{equation}\label{dfn_sigma}
	\Sigma' :=\cset{(f,v,\lambda)\in \cV}{\rank DF_f(v,\lambda)< n} = \cV\backslash\cW.
\end{equation}
Moreover, we denote by $\Sigma$ the Zariski closure of $\overline{\pi_1(\Sigma')}$. 
In reference to \cite{sturmfels-seigal}, we call $\Sigma$ the \emph{eigendiscriminant variety}.

\begin{rem}
We have $(f,v,\lambda)\in \Sigma'$ if and only if $(v,\lambda)$ is not an isolated root of the polynomial $F_f$. 
Thus, $f\in\pi_1(\Sigma')$ if and only if $F_f$ has a double root or $f$ has infinitely many roots.
\end{rem}


\begin{prop}\label{sigma_variety} 
\begin{enumerate}
\item \label{sigma_variety1} We have $f\not\in\Sigma$, if and only if the number of equivalence classes of $f$ equals $D(n,d)$.
\item \label{sigma_variety2}The set $\Sigma$ is a closed hypersurface of $(\cH_d)^n$ of degree at most~$n(n-1)d^{n-1}$.
\end{enumerate}
\end{prop}

\begin{proof}
For \Cref{sigma_variety1} use \cite[Theorem 1.2]{sturmfels-cartwright}. In \cite[Theorem 4.1, Corollary 4.2]{sturmfels-seigal} 
it is shown that the eigendiscriminant variety for tensors in $(\HC^n)^{\otimes d}$ is an irreducible hypersurface. 
We obtain $\Sigma$ by intersecting this with the linear subspace of symmetric tensors
and requiring $\|v\|=1$.
The assertion \Cref{sigma_variety2} follows from the dimension theorem, Bezout's theorem and the fact that 
$\Sigma$ is properly contained in $(\cH_d)^n$ (see \Cref{robeva_example} below).
\end{proof}
In \cite{robeva} the following explicit element in $(\cH_d)^n\backslash \Sigma$ is described ($d>1$). 
\begin{prop}\label{robeva_example}
Let $\phi(X) := \left(X_1^d,X_2^d,\ldots, X_n^d\right)\in(\cH_d)^n$. Then $\phi\not\in \Sigma$.
\end{prop}
\begin{proof}
One has
	\[F_\phi(X,\Lambda) = \begin{pmatrix}X_1^{d}-\Lambda  X_1\\ \vdots \\ X_n^{d}-\Lambda  X_n\end{pmatrix}.\]
We are going to show that $\phi$ has exactly $D(n,d)$ many classes of eigenpairs. 
Clearly, for any $v\in\HC\backslash\set{0}$ we have $F_\phi(v,0)\neq 0$. 
Hence, any equivalence class of eigenpairs of $\phi$ contains some representative of the form $(v,1)$. 
Let $\zeta$ be a primitive $(d-1)$-th root of unity and define 
\begin{equation*}
	M:=\Big\{(\epsilon_{1}\zeta^{i_1}, \ldots,\epsilon_{n}\zeta^{i_n}) \mid \epsilon\in\set{0,1}^n\backslash\set{0}, \forall j: 1\leq i_j\leq d-1 \Big\}
\end{equation*}
Observe that $F_\phi(z,1) = (z_i^{d}- z_i)_{i=1}^n = 0$, if and only if $z\in M\cup\set{0}$. For all $z\in M$ and $t\in\HC$ we have $(z,1)\sim(t z,1)$ 
if and only if $t=\zeta^i$ for some $1\leq i\leq d-1$. Let $\kU:=\langle \zeta\rangle$ denote the cyclic group generated by $\zeta$. 
We define a group action of $\kU$ on $M$ via componentwise multiplication. The number of equivalence classes of eigenpairs of $\phi$ 
then equals the number of $\kU \text{-orbits}$ in $M$. For $u\in\kU$ put $M^u:=\cset{z\in M}{uz=z}$. Observe that for $u\neq 1$ we have that $M^u=\emptyset$. 
Using Burnside's lemma we obtain
	\[\text{number of } \kU \text{-orbits in } M = \frac{1}{\lvert \kU\rvert} \,  \sum_{u\in\kU} \lvert M^u\rvert 
    = \frac{1}{\lvert \kU\rvert} \,  \lvert M\rvert = \frac{d^n-1}{d-1}=D(n,d).\]
\end{proof}

\subsection{The standard distribution on the solution manifold}

The definition of standard distribution is adapted from \cite[eq.~(17.19)]{condition}. Following \Cref{varphi_E},  
we say that a random variable $f$ on $(\cH_d)^n$ is standard normal distributed, if $f$ has the density 
\[
   \varphi(f):= \varphi_{(\cH_d)^n}(f)= \frac{1}{\pi^k}\, \exp\left(-\lVert f\rVert^2\right), \quad\text{ where } k=\dim_\HC (\cH_d)^n.
\]
By construction $\varphi(f)$ is invariant under the action of $\cU(n)$. 

The following procedure: 

\begin{enumerate}
\item choose $f$ according to the standard normal distribution.
\item choose some normalized eigenpair $(v,\lambda)$ of $f$ uniformly at random.
\end{enumerate} 
yields a probability distribution on $\cV$, which we call the \emph{standard distribution} 
and denote it by $(f,v,\lambda)\sim {\mathrm{STD}_\cV}$. 
Clearly, the standard distribution is invariant under the action of $\cU(n)$ on $\cV$. 

Observe that the two steps above are precisely the steps \Cref{main_distribution1}--\Cref{main_distribution2} in the operative description of $\rho^{n,d}$ 
given in the introduction.
This implies that $\rho^{n,d}$ equals the density of the pushforward measure of $\mathrm{STD}_\cV$ with respect to 
the projection
$\pi_3\colon\cV\to \HC,\, (f,v,\lambda)\mapsto \lambda$. 

According to \Cref{sigma_variety},  the fiber
\[
 V(f):=\cset{(v,\lambda)\in \HS(\HC^n)\times \HC}{(f,v,\lambda)\in \cV} = \pi_2(\pi_1^{-1}(f))
\]
over $f\not\in \Sigma$ consists of $D=D(n,d)$ disjoint circles, each of them having volume $2\pi$. Hence the density of the uniform distribution on $V(f)$ equals $(2\pi D)^{-1}$. As in \cite[Lemma~17.18]{condition}, one can now show that the density of the standard distribution is given by
\begin{equation}\label{dfn_standard_distr}
\rho_{\mathrm{STD}_\cV} (f,v,\lambda) = \frac{1}{2\pi D(n,d)}\, NJ(\pi_1)(f,v,\lambda) \, \varphi(f),
\end{equation}
where $\pi_1:\cV\to (\cH_d)^n$ is the projection from \eqref{projections1}. 

We denote by 
\[
 V(v,\lambda):=\cset{f\in (\cH_d)^n}{(f,v,\lambda)\in \cV} = \pi_1(\pi_2^{-1}(v,\lambda)) 
\]
the fiber of $\pi_2$ over $(v,\lambda)\in \HS(\HC^n)\times \HC$.

\begin{lemma}\label{useful_lemma}
Let $\theta: \cV\to \HR$ be an integrable map that is invariant under the group action from \Cref{groupaction} and $e_1:=(1,0,\ldots,0)\in\HS(\HC^n)$. Then
	\begin{equation*}
	\int_{(f,v,\lambda)\in \cV} \theta(f,v,\lambda) \, \rho_{\mathrm{STD}_\cV}(f,v,\lambda) \d \cV =
	\frac{\pi^{n-1}}{\Gamma(n) D(n,d)}\int\limits_{\lambda\in\HC}E(\lambda)\, \d \HC.
	\end{equation*}
where 
	\[E(\lambda)= \int\limits_{f\in V(e_1,\lambda)}  \left\lvert \det DF_f(e_1,\lambda) \right\rvert^2\, \theta(f,e_1,\lambda)\,  \varphi_{V(e_1,\lambda)}(f) \;\d V(e_1,\lambda).\]
\end{lemma}
\begin{proof}
Using the coarea formula, we obtain
\begin{align*}
&\int_{(f,v,\lambda)\in \cV} \theta(f,v,\lambda) \, \rho_{\mathrm{STD}_\cV}(f,v,\lambda) \d \cV\\
= &\int_{(v,\lambda)\in \HS(\HC^n)\times \HC} \left[\int_{f\in V(v,\lambda)} \frac{\theta(f,v,\lambda)\,  \rho_{\mathrm{STD}_\cV}(f,v,\lambda)}{NJ(\pi_2)(f,v,\lambda)} \d V(v,\lambda) \right] \d (\HS(\HC^n)\times \HC)
\end{align*}
By the definition of $\rho_{\mathrm{STD}_\cV}$, \Cref{quotient_NJ}, and the unitary invariance of $\theta$ we have that
\begin{align}
\nonumber&\int_{f\in V(v,\lambda)} \frac{\theta(f,v,\lambda)}{NJ(\pi_2)(f,v,\lambda)} \,  \rho_{\mathrm{STD}_\cV}\;\d V(v,\lambda) \\
\nonumber=&\;\frac{1}{2\pi D}\int_{f\in V(v,\lambda)} \lvert \det DF_f(v,\lambda)|_{T_v\times\HC} \rvert^2 \, \theta(f,v,\lambda)\,  \varphi(f) \;\d V(v,\lambda)\\
\nonumber=&\;\frac{1}{2\pi D}\int_{f\in V(e_1,\lambda)} \lvert \det DF_f(e_1,\lambda)|_{T_{e_1}\times\HC} \rvert^2 \, \theta(f,e_1,\lambda)\,  \varphi(f) \;\d V(e_1,\lambda)\\
=&\; \frac{E(\lambda)}{2\pi D}.\label{aa1}
\end{align}
Observe that the integral \Cref{aa1} does not depend on $v$ anymore. The claim follows by using $\int 1 \d \HS(\HC^n)=\frac{2\pi^n}{\Gamma(n)}$ 
\end{proof}


\section{Proofs}\label{se:proofs}

We are now ready to prove \Cref{main_thm}. 

\begin{prop}\label{pushforward}
The pushforward density of ${\mathrm{STD}_\cV}$ with respect to $\pi_3$ is
	\begin{align*}
	\rho^{n,d}(\lambda) &= \frac{d^{n-1}\, e^{-\lvert\lambda\rvert^{2}\left(1-\frac{1}{d}\right)}}{\pi D(n,d)}\,
	  \frac{\Gamma\left(n,\frac{\lvert \lambda \rvert^{2}}{d}\right)}{\Gamma(n)} = \frac{d^{n-1}\,  e^{-\lvert\lambda\rvert^{2}}}{\pi D(n,d)}\, 
	  \sum_{k=0}^{n-1}\frac{1}{k!}\, \left(\frac{\lvert \lambda \rvert^{2}}{d}\right)^k
	 \end{align*}
\end{prop}
\begin{proof}
Before we start, we remark that \Cref{auxiliary_lemma} justifies the right equality. By \Cref{useful_lemma}, the pushforward distribution $\rho^{n,d}(\lambda)$ is obtained by computing
	\begin{equation}\label{pushforward_integral}
	\frac{\pi^{n-1}}{\Gamma(n) D(n,d)} \int_{f\in V(e_1,\lambda)}  \left\lvert \det DF_f(e_1,\lambda)|_{T_{e_1}\times\HC} \right\rvert^2  \varphi_{V(e_1,\lambda)}(f) \d V(e_1,\lambda)
	\end{equation}
The case $n=1$ is an easy exercise. So let us assume that $n>1$. Observe that $V(e_1,\lambda)$ is the affine space
	\[V(e_1,\lambda) = \lambda X_1^d e_1 + \Big\{g\in(\cH_d)^n \mid g(e_1)=0 \Big\}.\]
Let $R:=\Big\{h\in (\cH_d)^n \mid h(e_1)=0, Dh(e_1)=0 \Big\}$. 
By \cite[equation (16.10)]{condition}, for any $f\in V(e_1,\lambda)$, there exist uniquely determined $h\in R$ and $M\in \HC^{n \times (n-1)}$ 
such that we can orthogonally decompose $f$ as
	\begin{equation}\label{aa2}
	f = \lambda X_1^d e_1 + X_1^{d-1}\sqrt{d} M \,  X' +h,
	\end{equation}
where $X'=(X_2,\ldots,X_n)^T$. We have that 
	\begin{equation}
	\partial_X f(e_1,\lambda) = \begin{bmatrix} \partial_{X_1}f(e_1,\lambda), & \partial_{X'}f(e_1,\lambda)\end{bmatrix}
    = \begin{bmatrix}  d\lambda e_1, & \sqrt{d} M\end{bmatrix}\in \HC^{n\times n}\label{jacobian_f}
	\end{equation}
Let $a\in\HC^{1 \times (n-1)}$ be the first row of $M$ and $A\in
\HC^{(n-1) \times (n-1)}$ be the matrix that is obtained by removing
the first row of $M$. 
By \Cref{jacobian} and \Cref{jacobian_f} the derivative of $F_f$ at $(e_1,\lambda)$ has the matrix representation
	\begin{equation*}
	\begin{bmatrix} \partial_X f(e_1,\lambda) - \lambda I_{n},&
          -e_1\end{bmatrix} = \begin{bmatrix} (d-1)\lambda &
          \sqrt{d}\, a & -1 \\ 0 & \sqrt{d} A-\lambda I_{n-1} & 0 \end{bmatrix} \in\HC^{n\times (n+1)}.
	\end{equation*}
This implies $\det DF_f(e_1,\lambda)|_{T_{e_1}\times \HC} = -\det\;(\sqrt{d} A-\lambda I_{n-1}).$ 

The summands in \Cref{aa2} are pairwise orthogonal. From this we get that $\lVert f\rVert^2 = \lvert \lambda \rvert^2+\lVert M\rVert_F^2+~\lVert h\rVert^2$, which implies that
	\[\varphi_{V(e_1,\lambda)}(f) = \frac{1}{\pi^{n}} \, e^{-\lvert \lambda\rvert^ {2}}\, \varphi_{\HC^{(n-1)\times (n-1)}}(A)\, \varphi_{\HC^n}(a)\, \varphi_R(h).\]
Integrating over $a$ and $h$ in \Cref{pushforward_integral} therefore yields
	\begin{align*}
	&\int_{f\in V(e_1,\lambda)}  \left\lvert \det DF_f(e_1,\lambda)|_{T_{e_1}\times \HC} \right\rvert^2  \varphi_{V(e_1,\lambda)}(f) \d V(e_1,\lambda)\\
	=\:& \frac{e^{-\lvert\lambda\rvert^{2}}}{\pi^n} \,
        \mean\limits_{A\sim N(0,\frac{1}{2}I_{n-1})} \left\lvert\det\left(\sqrt{d}\, A-\lambda I_{n-1}\right) \right\rvert^2 \\
	=\:& \frac{d^{n-1} e^{-\lvert\lambda\rvert^{2}}}{\pi^n} \,
        \mean\limits_{A\sim N(0,\frac{1}{2}I_{n-1})}
         \left\lvert \det\left(A-\frac{\lambda}{\sqrt{d}}I_{n-1}\right)\right\rvert^2 \\
	=\:& \frac{d^{n-1}}{\pi^n} \,  e^{-\lvert \lambda\rvert^{2}(1-\frac{1}{d})}\,  \Gamma\left(n,\frac{\lvert \lambda \rvert^{2}}{d}\right);
	\end{align*}
the last line by \Cref{expectation_det}. Plugging this into \Cref{pushforward_integral} the claim follows.
\end{proof}

\begin{proof}[Proof of \Cref{main_thm}]
\Cref{pushforward} shows that the distribution of the eigenvalue $\lambda$ only depends on $\lvert \lambda\rvert$. As in \Cref{R} we put~$r:=\lvert \lambda \rvert$ and $R:=2 r^2$. Making a change of variables, we obtain the density $\rho^{n,d}_\HR(R):=\frac{\pi}{2}\, \rho^{n,d}(r)$. From \Cref{pushforward} we obtain
	\begin{equation}\label{radial_distr}
	\rho^{n,d}_\HR(R) = \frac{ d^{n-1}}{2 D(n,d)}\, e^{-\frac{R}{2}}\,
	  \sum_{k=0}^{n-1}  \frac{1}{k!} \left(\frac{R}{2d}\right)^k.
	  \end{equation}	  
If $d=1$, \Cref{radial_distr} becomes
\begin{equation*}
	\rho^{n,1}_\HR(R) =   \frac{1}{n}\, \sum\limits_{k=0}^{n-1}\frac{e^{-\frac{R}{2}} R^k}{2^{k+1} k!}  
       =   \frac{1}{n}\, \sum\limits_{k=1}^{n}  \frac{e^{-\frac{R}{2}} R^{k-1}}{2^{k} (k-1)!}.
\end{equation*}
For any $k$ we have that $e^{-\frac{R}{2}} \, R^{k-1}/(2^{k} (k-1)!)$ is the density of a chi-square distributed random variable with $2k$ degrees of freedom, 
which proves the assertion in this case.

If $d>1$, put $q:=\frac{1}{d}$, such that $D(n,d)=(1-q^n)/(q^{n-1}(1-q))$. Then \Cref{radial_distr} becomes
\begin{equation*} 
	\rho^{n,d}_\HR(R)=\sum\limits_{k=0}^{n-1}    \frac{e^{-\frac{R}{2}}\, R^k}{2^{k+1} k!}\, \frac{(1-q)q^k}{1-q^n} 
        =\sum\limits_{k=1}^{n}    \frac{e^{-\frac{R}{2}}\, R^{k-1}}{2^{k} (k-1)!}\, \frac{(1-q)q^{k-1}}{1-q^n}.
  \end{equation*}
Using  that $\prob\limits_{X\sim \mathrm{Geo}(1-q)}\cset{X=k}{X\leq n} =q^{k-1}(1-q)/(1-q^n)$, see \Cref{truncated_dfn}, finishes the proof.
\end{proof}
\label{sec_second}

To prove \Cref{second_main_result} we will need the following lemma.
\begin{lemma}\label{main_cor}
Let $n\geq 1$.
\begin{enumerate}
\item If $d=1$, then $\mean\limits_{\lambda \sim \rho^{n,1}}\lvert \lambda \rvert^2 = \mean\limits_{X\sim \mathrm{Unif}(\set{1,\ldots,n})}[X],$  
\item If $d>1$, then $\mean\limits_{\lambda \sim \rho^{n,d}}\lvert \lambda \rvert^2 = \mean\limits_{X\sim \mathrm{Geo}(1-\frac{1}{d})}[X\mid X\leq n].$ 
\end{enumerate}
\end{lemma}
\begin{proof}
We prove the claim for $d>1$. (The case $d=1$ is proven similarly.)
If $R=2\lvert\lambda\rvert^2$, then $\mean\limits_{\lambda \sim \rho^{n,d}}[\lvert \lambda \rvert^2]= \frac{1}{2}\mean[R]$. From \Cref{main_thm} we get
	\begin{align*}
	\mean[R] &= \int_{R=0}^\infty  R\, \rho^{n,d}_\HR(R) \d R\\
	 &= \sum_{k=1}^{n} \prob\limits_{X\sim \mathrm{Geo}(1-\frac{1}{d})}\cset{X=k}{X\leq n} \, \int_{R=0}^n R \,  \chi^2_{2k}(R)\d R.\\
	 &= \sum_{k=1}^{n} \prob\limits_{X\sim \mathrm{Geo}(1-\frac{1}{d})}\cset{X=k}{X\leq n}\, 2k\; =\;  2\mean\limits_{X\sim \mathrm{Geo}(1-\frac{1}{d})}[X\mid X\leq n],
	\end{align*}
where we have used that a $\chi_{2k}^2$-distributed random variable with $2k$ degrees of freedom has the expectation $2k$. 
\end{proof}

\begin{proof}[Proof of \Cref{second_main_result}]
If $d=1$, from \Cref{main_cor} we immediately get $\mean\limits_{\lambda \sim \rho^{n,1}}\lvert \lambda \rvert^2 = \frac{n+1}{2}.$

If $d>1$, by  \Cref{main_cor}, we have that $\mean\limits_{\lambda \sim \rho^{n,d}}\lvert \lambda \rvert^2= \mean\limits_{X\sim \mathrm{Geo}(1-\frac{1}{d})}[X\mid X\leq n].$ Therefore, \Cref{expectation_truncatedgeo} with $q:=\frac{1}{d}$  implies
\[\mean\limits_{\lambda \sim \rho^{n,d}}\lvert \lambda \rvert^2 =  \frac{n -(n+1)d +d^{n+1}}{(d^n-1)(d-1)}\]
as claimed. For fixed $n$ we obtain 
\[\lim\limits_{d\to 1}  \frac{n-(n+1)d +d^{n+1}}{(d^n-1)(d-1)}= \frac{n+1}{2}\]
by using de l'Hopital's rule twice. Therefore, the map
	\[\HR_{\geq 1} \to \HR,\quad d\mapsto \begin{cases} \frac{n+1}{2}, & \text{if } d=1 \\ \frac{n-(n+1)d +d^{n+1}}{(d^n-1)\,(d-1)} , & \text{if } d>1 \end{cases}.\]
is continous and differentiable on $\HR_{>1}$. One checks that its derivative on~$\HR_{>1}$ is negative. Hence, for fixed $n$, we see that $d\mapsto \mean\limits_{\lambda \sim \rho^{n,d}}[\lvert \lambda \rvert^2] $ is strictly decreasing. In the same way we can prove that, if $d$ is fixed, $n\mapsto \mean\limits_{\lambda \sim \rho^{n,d}}[\lvert \lambda \rvert^2] $ is strictly increasing. Further, 
\[\lim\limits_{d\to \infty} \frac{n-(n+1)d +d^{n+1}}{(d^n-1)\,(d-1)} = \lim\limits_{q\to 0}  \frac{nq^{n+1} -(n+1)q^n +1}{(1-q^n)\,(1-q)} =1\]
If $d> 1$, we have 
	\[\lim\limits_{n\to \infty} \mean\limits_{\lambda \sim \rho^{n,d}}[\lvert \lambda \rvert^2] = \lim\limits_{n\to \infty} \frac{nq^{n+1} -(n+1)q^n +1}{(1-q^n)\,(1-q)} = \frac{1}{1-q},\]
where again $q=\frac{1}{d}$.
\end{proof}


\label{sec_third}
\begin{proof}[Proof of \Cref{third_main_result}]
Let $d>1$. Recall from \Cref{tau} that we have put $\tau= \frac{R(d-1)}{4d}$ and that we denote the density of $\tau$ by $\rho^{n,d}_\textbf{norm}$. Using \Cref{main_thm}  we get 
	\begin{equation*}
	\rho^{n,d}_\textbf{norm}(\tau) = \frac{4 d}{d-1}\, \rho^{n,d}_\HR\left(\frac{4d\tau}{d-1}\right)\nonumber 
          = \frac{2 d^n}{d^n-1}\, e^\frac{-2d\tau}{d-1}\, \sum_{k=0}^{n-1}   \frac{1}{ k!} \left(\frac{2\tau}{d-1}\right)^k.
	\end{equation*} 
Again putting $q=\frac{1}{d}$ we obtain 
		\begin{equation*}
	\rho^{n,d}_\textbf{norm}(\tau) = \frac{2}{1-q^n}\, e^\frac{-2\tau}{1-q}\, \sum_{k=0}^{n-1}   \frac{1}{ k!} \left(\frac{2q\tau}{1-q}\right)^k.
	\end{equation*} 
Since $0<q<1$, we have $\lim\limits_{n \to \infty} q^n = 0$. Hence,  
	\begin{equation*}
	\lim_{n\to \infty} \rho^{n,d}_\textbf{norm}(\tau) =   2\, e^\frac{-2\tau}{1-q}\, \sum_{k=0}^{\infty}   \frac{1}{ k!} \left(\frac{2q\tau}{1-q}\right)^k 	 
         = 2\, e^{-2\tau}, 
	 \end{equation*}
which finishes the proof.
\end{proof}

{
\bibliographystyle{plain}
\bibliography{literature}}
Plots: MATLAB R2015b
\end{document}